\documentclass[a4paper,11pt]{article}

\usepackage{amsmath}
\usepackage[textheight=52pc,textwidth=35pc]{geometry}

\usepackage[utf8]{inputenc}
\usepackage[USenglish]{babel}
\usepackage[numbers,sort&compress]{natbib}
\usepackage[colorlinks,citecolor=blue,urlcolor=blue,linkcolor=blue]{hyperref}
\usepackage[capitalize]{cleveref}
\crefformat{equation}{(#2#1#3)}
\crefname{algocf}{Alg.}{Algs.}
\usepackage{autonum}
\usepackage{csquotes}
\usepackage{mathtools}
\usepackage{bbm}
\usepackage{amsfonts,amssymb,amsthm}
\usepackage{nicefrac}
\usepackage[dvipsnames]{xcolor}

\usepackage[vlined,linesnumbered,ruled]{algorithm2e}
\let\oldnl\nl
\newcommand{\nonl}{\renewcommand{\nl}{\let\nl\oldnl}}
\makeatother

\usepackage{placeins}
\usepackage[singlelinecheck=false]{subfig}
\usepackage{colortbl}

\usepackage{xcolor}

\theoremstyle{plain}
\newtheorem{theorem}{Theorem}
\newtheorem{proposition}{Proposition}
\newtheorem{corollary}{Corollary}
\theoremstyle{definition}
\newtheorem{definition}{Definition}
\newtheorem{example}{Example}
\theoremstyle{remark}
\newtheorem{remark}{Remark}

\newcommand{\St}{\mathcal{S}}
\newcommand{\T}{\mathcal{T}}
\newcommand{\floretfun}[1]{E_{#1}}
\newcommand{\edgefun}[1]{E_{{#1}}}
\newcommand{\edgeset}{E}
\newcommand{\baretree}{T=(V,\edgeset)}

\newcommand{\st}{(S,\theta')}
\newcommand{\pt}{(T,\theta)}
\newcommand{\boldtheta}{{\boldsymbol\theta}}
\newcommand{\pithetalambda}{\pi_\theta(\lambda)}
\newcommand{\cT}{c_{\T}}
\newcommand{\summ}[1]{\sum_{\mathclap{#1}}}

\newcommand{\poly}[1]{{\mathrm{poly}(#1)}}
\newcommand{\G}{\mathcal{G}}
\newcommand{\ZZ}{\mathbbm Z}

\newcommand{\ideal}[1]{\langle #1 \rangle}
\DeclareMathOperator{\supp}{support}

\DeclareMathOperator{\Val}{Val}

\title{Discovery of statistical equivalence classes using computer algebra}
\author{Christiane G\"orgen, Anna Bigatti, Eva Riccomagno and Jim Q.~Smith}

\begin{document}
\maketitle
  
\begin{abstract} 
Discrete statistical models supported on labelled event trees can be specified using so-called interpolating polynomials which are generalizations of generating functions.
These admit a nested representation. 
A new algorithm exploits the primary decomposition of monomial ideals associated with an interpolating polynomial to quickly compute all nested representations of that polynomial. It hereby determines an important subclass of all trees representing the same statistical model.
To illustrate this method we analyze the full polynomial
equivalence class of a staged tree representing the best fitting model inferred from a
real-world dataset.\\

\end{abstract}

\section{Introduction} 
Families of finite and discrete multivariate models have been extensively studied, including
many different classes of graphical models~\citep{lauritzen,amp:equiv}. Because these families of probability distributions can often be expressed as polynomials -- or collections of vectors of polynomials -- this has spawned a deep study of their algebraic properties~\citep{pistonerw:algstat,sturmfels:bio,oberwolfach}.  These
can then be further exploited using the discipline of computational commutative algebra and computer algebra software such as \texttt{CoCoA}~\citep{cocoa} which has proved to be a powerful though somewhat neglected tool of analysis. 

In this paper, we demonstrate how certain computer algebra techniques -- especially the primary decomposition of ideals -- can be routinely applied to the study of various finite discrete models. Throughout we pay particular attention to an important class of graphical models based on probability trees and called \emph{staged trees} or \emph{chain event graph} models~\citep{ours}. These contain the familiar class of discrete (and context-specific) Bayesian networks as a special case.  
In particular,~\cite{meins} gave a mathematical way of determining the statistical equivalence classes of staged tree models but did not give algorithms to actually find these. Here we use computer algebra in a novel way to systematically find a staged tree representation of a  given family -- if it indeed exists -- and to uncover statistically equivalent staged trees in an elegant, systematic and useful way. This is an extensions of the techniques developed by~\cite{amp:equiv} and others to determine Markov-equivalence classes of Bayesian networks where, instead of algebra,  graph theory was used as a main tool. 

So our methodology supports a new analysis of a very general but fairly recent statistical model class in a novel algebraic way and serves as an illustration of how more generally computer algebra can be a useful tool not only to the study of conventional classes of graphical model but other families of statistical model as well.

\section{Staged trees and  interpolating polynomials}
\label{sect:graphtopoly}

\subsection{Labeled event trees and staged trees}\label{sub:trees}

In this work we will exclusively consider graphs which are
trees, so those which are connected and without cycles. 
We first review the theory of staged trees which represent
interesting and very general discrete models in statistics~\citep{ours}.

\begin{definition}[Labeled event trees]\label{def:tree}
Let $\baretree$ be a finite directed rooted tree with vertex set $V$ and edge set $\edgeset\subseteq V\times V$.
We denote the root vertex {of $T$} by $v_0$.

The tree $T$ is called an \emph{event tree} if every vertex {$v\in V$} has either no, two or more than two emanating edges.   
For $v\in V$, let 
$\floretfun{v} =\{(v,w) \mid w\in V\} \cap \edgeset$
denote the set of the edges emanating from $v$.
The pair $(v,\floretfun{v})$ is called a \emph{floret}.

Let $\Theta$ be a non-empty set of symbol/labels and let a function
$\theta:\edgeset\longrightarrow\Theta$ be such that for any floret
$(v,\floretfun{v})$ the labels in $\theta(\floretfun{v})$ are all distinct.
We call $\theta(\floretfun{v})$ the \emph{floret labels} of $v$
and denote this set by $\theta_v$. 
The pair $\T=\pt$ of graph and function is called a \emph{labeled event tree}. 
When $\theta$ takes values in $(0,1)$ and $ \sum_{e\in \floretfun{v}}
\theta(e)=1$,  $\T$ is called a \emph{probability tree}\footnote{We
  should say more precisely: when the symbols $\theta(e)$ are
  evaluated in $(0,1)$ for all $e\in E$.
   }. 

For $v\in V$, the \emph{labeled subtree rooted in $v$} is
   $\T_v = (T', \theta')$, where $T'$ is the largest
  subtree of $T$ rooted in $v$, and $\theta'$ is 
   the restriction of $\theta$ to the edges in $T'$.
\end{definition}

For any leaf $v\in V$, so for any vertex with no emanating edges, we trivially have that $\edgefun{v}=\emptyset$, and hence $\theta_v=\emptyset$. 

labeled event trees are well-known objects in probability theory and decision theory where they are used to depict discrete unfoldings of events. The labels on edges of a probability tree then correspond to transition probabilities from one vertex to the next and all edge probabilities belonging to the same floret sum to unity. See~\cite{shafer:causalconj} for the use of probability trees in probability theory and causal inference, and see for instance~\cite{salmeron} for how such a tree representation can be used in computational statistics.

In this paper, we generally do not require the labels on a labeled event tree to be probabilities.

\begin{definition}[Staged trees]\label{def:treestaged}
A labeled event tree $\T=\pt$, with $\baretree$, is called a \emph{staged tree} if
for every pair of vertices $v, w\in V$ 
their floret labels are either equal or disjoint,
$\theta_v=\theta_w$ or $\theta_v \cap \theta_w=\emptyset$.
A \emph{stage} is a set of vertices with the same floret labels.
\end{definition} 

In illustrations of staged trees, all vertices in the same stage are usually assigned a common \emph{color}: compare \cref{fig:labeledevent}. Staged trees were first defined as an intermediate step to building \emph{chain event graphs} as
graphical representations for certain discrete statistical models~\cite{smithanderson:ceg}. Every chain event graph is
uniquely associated to a staged tree and vice versa. In this way, the
graphical redundancy of staged trees can be avoided, and elegant conjugate analyzes can be applied to staged tree models
\citep{thwaitessmithcowell:propagation,freeman:MAP,lorna:BNceg,rodrigojim}. In particular, every discrete and context-specific Bayesian network can alternatively be represented by a staged tree where stages indicate equalities of conditional probability vectors. We give examples of this later in the text.

For the development in this paper it is important to observe that staged trees with labels evaluated as probabilities are always also probability trees. This is however not the case for all labeled event trees because sum-to-$1$ conditions imposed on florets can be contradictory. See also Examples~\ref{bsp:labeledevent} and~\ref{bsp:skeleton}  below.

\begin{example}[Saturated trees]\label{bsp:saturated} 
A \emph{saturated} tree is a labeled event tree where all edges have {distinct}
labels. So this is a staged tree where all floret labels are
disjoint, or alternatively with every stage containing exactly one vertex. In the
development below, saturated trees are graphical representations of
saturated statistical models.
\end{example}

\begin{example}\label{bsp:labeledevent}
\Cref{fig:notstaged_staged} shows a staged tree where all blue-coloured
vertices are in the same stage. 
\Cref{fig:skeleton} depicts a staged tree where the two green vertices are in the same stage.
\Cref{fig:notstaged} show a labeled event tree which is not staged because 
the floret labels of the two black vertices are  neither equal nor disjoint. 
\end{example}

\begin{figure}
\centering
\subfloat[A labeled event tree which is staged.]{
\centering
\includegraphics[scale=1]{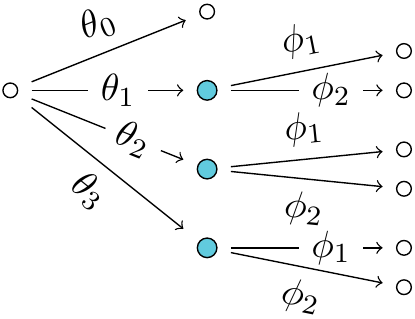}
\label{fig:notstaged_staged}}
~~
\subfloat[Another labeled event tree which is staged.]{
\centering
\raisebox{11pt}{\includegraphics[scale=1]{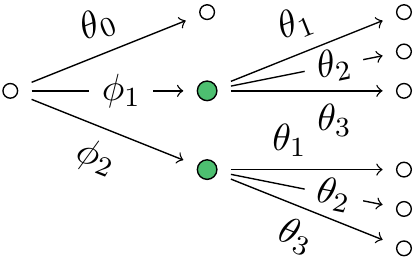}}
\label{fig:skeleton}}
~~
\subfloat[A labeled event tree which is not staged.]{
\centering
\includegraphics[scale=1]{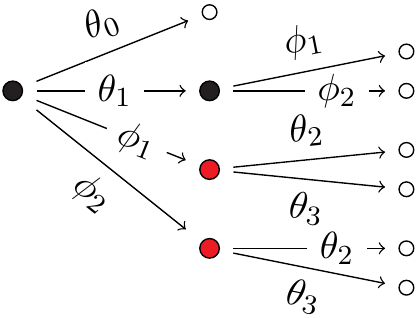}
\label{fig:notstaged}}
\caption{Three illustrations of labeled event trees, analyzed in Examples~\ref{bsp:labeledevent}, \ref{example1poly2factor2trees} and~\ref{bsp:skeleton}.}\label{fig:labeledevent}
\end{figure}

\subsection{Network polynomials and interpolating polynomials}
\label{sec:network-interpol}

We next define a polynomial associated to a labeled event tree which is the key tool used in this paper: see also~\cite{meins}.

\begin{definition}[Network and interpolating polynomials]\label{def:interpol}
Let $\T = \pt$ be a labeled event tree and let  $\Lambda(\T)$ denote the set of  root-to-leaf paths in $\T$.
For $\lambda\in \Lambda(\T)$ let $\edgefun{\lambda}$ be the set of edges of
$\lambda$. We call the products of the labels along a root-to-leaf path,
$\pithetalambda=\prod_{e\in \edgefun{\lambda}}\theta(e)$, \emph{atomic monomials}.

Given a real-valued function $g:\Lambda(\T)\rightarrow \mathbb{R}$, 
we define the \emph{network polynomial} of $\T$ and $g$, the linear combination of the atomic monomials with coefficients given by $g$, as:
\begin{equation}\label{eq:networkpol}
c_{g,\T}=\summ{\lambda\in\Lambda(\T)}\;g(\lambda)\cdot\pithetalambda
\end{equation}
with the particular case $c_{g,\T}=1$ if   $\T$ has no edges.
The \emph{interpolating polynomial} is the network polynomial with all $g(\lambda)=1$ equal to one, and we write $\cT=c_{1,\T}$.  
\end{definition}

\begin{remark}
A \emph{network polynomial} $c_{g,\T}$ is a polynomial in the ring $\mathbb R[\Theta]$ of polynomials with real coefficients and
whose indeterminates are the labels in $\Theta$.
An \emph{interpolating polynomial}  $\cT$ is
a polynomial with positive integer coefficients by construction. For these we write $ \cT \in  \ZZ[\Theta]$. 
\end{remark}

\begin{example} 
When $\T=\pt$ is a probability tree,  every atomic monomial $\pithetalambda$ is the product of transition probabilities along a root-to-leaf path and thus the probability of an \emph{atomic event} (or atom). 
Often the function $g$ is an indicator function
$g=\mathbbm{1}_{A}$ of an event $A\subseteq\Lambda(\T)$.
In this case, \cref{eq:networkpol} is a polynomial representation of the
finite-additivity property of probabilities for $A$, 
so $c_{\mathbbm{1}_{A},\T}=\sum_{\lambda\in A}\pithetalambda$.
\end{example}

Interpolating polynomials have been used successfully to classify equivalence classes of staged trees which make the same distributional assumptions~\citep{meins}, as outlined in \cref{sub:equivalence} below. They have further been used as a tool  for calculating marginal and conditional probabilities in Bayesian networks and staged trees, using differentiation operations~\citep{darwiche:networkpol,manuich}.

In \cref{prop:nesting1} and \cref{prop:nesting2} for the purposes of this paper we now present two central results on interpolating polynomials. These results are given here in a reformulated, recursive form and very different from their original development \cite[Proposition 1]{meins}. This refinement is necessary because the new proofs we give are constructive and, most importantly, transparently illustrate the mechanisms needed for our later algorithmic implementation.

\begin{theorem} 
\label{prop:nesting1}
Let $\T=\pt$ 
be an event tree and for $v\in V$ define 
\begin{equation}
\poly{\T_v}=\begin{cases}
1 & \text{ if } \floretfun{v}=\emptyset \\
\sum_{(v,w)\in \floretfun{v}} \theta(v,w)\cdot \poly{\T_w} &  \text{ otherwise. }
\end{cases}
\end{equation}
Then the interpolating polynomial $\cT$ of $\T$ is equal to $\poly{\T_{v_0}}$ where $v_0$ is the root of $\T$. 
\end{theorem}

\begin{proof}
We prove the claim by induction on the \emph{depth} of the tree,
i.e.\,the number of edges in the longest root-to-leaf path. 
If $\T$ has depth $=0$ then $\floretfun{v_0} =\emptyset$ and $c_\T = 1 = \poly{\T_{v_0}}$. 
If $\T$ has depth $\ge1$ then
\begin{equation}
\poly{\T_{v_0}}
=\summ{(v_0,w)\in \floretfun{v_0}}\theta(v_0,w)\cdot \poly{\T_w}.
\end{equation}
Furthermore, 
\begin{equation}
\cT =
\summ{\lambda\in\Lambda(\T)}\pithetalambda
\;=\;
\summ{(v_0,w)\in \floretfun{v_0}}
\theta(v_0,w)\cdot 
\summ{\lambda' \in \Lambda(\T_{w})} \pi_\theta(\lambda')
\;=\;
\summ{(v_0,w)\in \floretfun{v_0}}
\theta(v_0,w) \cdot c_{\T_{w}}
\end{equation}
and  $\poly{\T_{w}} = c_{\T_{w}}$
by the inductive hypothesis because the subtrees $\T_w$ all have lower depths than $\T$.
\end{proof}

\begin{figure}[tb]
\centering
\subfloat[A labeled event tree $\pt$.]{
\includegraphics[scale=1]{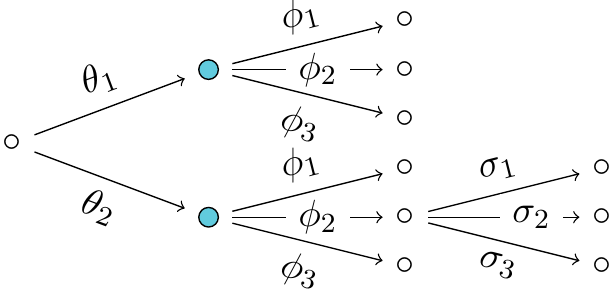}\label{fig:nestinga}}
\qquad
\subfloat[A labeled event tree $\st$.]{
\includegraphics[scale=1]{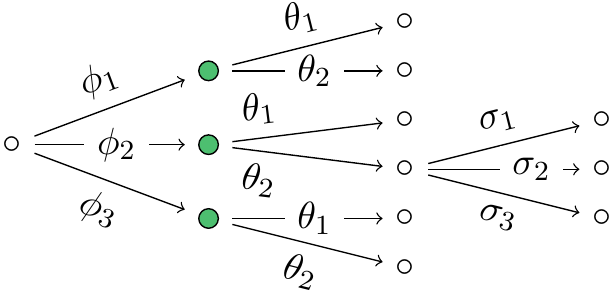}\label{fig:nestingb}}
\caption{
Two staged trees with the same interpolating polynomial but different graphs.
See Examples~\ref{ex:running0} and~\ref{ex:running2} and 
\cref{eq:nestingbsp1} and~(\ref{eq:nestingbsp2}) in
\cref{ex:running1}.}
\label{fig:nesting}
\end{figure}

\begin{example}\label{ex:running0}
The two staged trees $\T=\pt$ and $\St=\st$ in \cref{fig:nesting} 
have the same interpolating polynomial, so the same sum of atomic
monomials:
\begin{equation}\label{eq:exatomic}
\cT
=c_{\St}
=\theta_1\phi_1+\theta_1\phi_2 +\theta_1\phi_3 
+\theta_2\phi_1 
+\theta_2\phi_2\sigma_1+\theta_2\phi_2\sigma_2+\theta_2\phi_2\sigma_3 
+\theta_2\phi_3.
\end{equation}
Here, the functions $\theta$ and $\theta'$ assign the same labels to different edges in the graphs $T$ and $S$. 
Following the recursive construction in \cref{prop:nesting1}, we can then write
this polynomial in terms of the interpolating polynomials of subtrees:
\begin{equation}\label{eq:exrootexpct}
\cT=\poly{\T}=
\theta_1\cdot \poly{\T_1}
+ \theta_2\cdot \poly{\T_2}
\end{equation}
where 
 $\poly{\T_1}=\phi_1+\phi_2$ and 
 $\poly{\T_2}=\phi_1+\phi_2\cdot(\sigma_1{+}\sigma_2{+}\sigma_3)+\phi_3$; 
or alternatively
\begin{equation}\label{eq:exrootexpcs}
c_{\St}=\poly{\St}=   
\phi_1\cdot \poly{\St_1}
+ \phi_2\cdot \poly{\St_2}
+ \phi_3\cdot \poly{\St_3}
\end{equation}
where 
$\poly{\St_1}=\theta_1+\theta_2$,
$\poly{\St_2}=\theta_1+\theta_2 \cdot(\sigma_1{+}\sigma_2{+}\sigma_3)$ 
and 
$\poly{\St_3}=\theta_1+\theta_2$.
\end{example}

\Cref{ex:running0} shows that the distributive property of multiplication over addition is at the core of our work. 
The following corollary will be useful for studying staged trees with 
square-free atomic monomials: compare also \cref{prop:rootfloretlabels} below.

\begin{corollary}\label{cor:rootexpansion}
Let $\T=\pt$ be a labeled event tree and 
let $\cT$ be its interpolating polynomial. 
Then  we can write
\begin{equation}\label{eq:rootexpansion}
\cT=\summ{ (v_0,w)\in \floretfun{v_0} } \theta(v_0,w) \cdot c_{\T_w}.
\end{equation}
Moreover, if the root labels are not repeated, i.e.~$\theta_{v_0} \cap \theta_v=\emptyset$ for all $v\in
V{\setminus} \{ v_0\}$, then no label in $\theta_{v_0}$
appears in any subtree-interpolating polynomial $c_{\T_w}$.
\end{corollary}

\begin{proof} 
The proof is a trivial consequence of the construction of 
the polynomial $\poly{\T_{v_0}}$ in \cref{prop:nesting1} above.
\end{proof} 

\begin{example}\label{ex:running2}
 Consider again the two staged trees in \cref{ex:running0}. 
Their   interpolating polynomial admits two different representations in
terms of a linear combination as in \cref{cor:rootexpansion}, namely the ones in \cref{eq:exrootexpct} and~\cref{eq:exrootexpcs}.
We can see here explicitly how the polynomials above depend on the variables in subtrees of $\pt$ and $\st$. In particular, both sets 
$\{\theta_1,\theta_2\}$ and $\{\phi_1,\phi_2,\phi_3\}$
provide potential root-floret labels of a corresponding tree representation.
\end{example}

\subsection{Polynomials with a nested representation}

We know now that we can straightforwardly read an interpolating polynomial, and in particular a recursive representation of that polynomial, from a labeled event tree. In this section and in \cref{sect:treecomp} we consider the inverse problem: given a polynomial in distributed form can we tell whether it is the interpolating polynomial of a labeled event tree? 
In order to answer this question first observe that the polynomials defined below admit a special structured representation and can be used as a surrogate for a labeled event tree as shown in \cref{prop:nesting2}.

\begin{definition}[Nested representation]\label{def:nestedpoly}
Let $f\in\ZZ[\Theta]$ be a polynomial with positive integer coefficients.  We say that $f$ admits a \emph{nested representation} if $f=1$ or if it can be written as $f=\sum_{x\in A}x\cdot f_x$ where $A\subseteq\Theta$ is such that $\# A\geq 2$ and, for each $x\in A$, the polynomial $f_x$ admits a nested representation.
\end{definition}

\begin{remark}\label{rem:positivecoefficient}
The recursion in \cref{def:nestedpoly} is finite because $\deg(f_x) = \deg(f)-1$, for by construction polynomials with nested representations have positive coefficients.
\end{remark}

The polynomial $\poly{\T_v}$ in \cref{prop:nesting1} is written in nested representation by construction. 
In this sense \cref{prop:nesting2} below is the inverse result of \cref{prop:nesting1}, and a polynomial admits a nested representation if and only if it is the interpolating polynomial of a labeled event tree.

\begin{proposition}\label{prop:nesting2}
If $f\in\ZZ[\Theta]$ admits  a nested representation then there exists a labeled event tree $\T$ such that $f = c_\T$.
\end{proposition}
\begin{proof}
We prove the claim by induction on the degree of $f$. 
If $\deg(f)=0$ then $f=1$ and therefore $f=c_\T$ where $\T$ is formed by a single vertex with no edges and no labels.

If $\deg(f)>0$ then $f=\sum_{x\in A} x \cdot f_x$ and therefore by 
\cref{rem:positivecoefficient} and by induction $f_x=c_{\T_x}$ for
some tree $\T_x$ labeled over $\Theta$. For all $x\in A$ let $v_x$ be the root of $\T_x$. Then a tree $\T$ with interpolating
polynomial $f$ can be constructed by taking a new vertex $v_0$
assigned as the root of $\T$ and defining the edges of the root
floret $\floretfun{v_0}$ to be $\{(v_0, v_x)\mid x\in A\}$. 
Then $f=c_\T$.
\end{proof}
 
The result above implies in particular that if $f$ is a polynomial with nested representation $f=\sum_{x\in A} x \cdot f_x$ then the root labels of a tree with interpolating polynomial $f$ are given by  $A$. 

\begin{example}
\label{ex:running1}
The nested representations of the two event trees $\T$ and $\St$ in
\cref{fig:nesting} are
\begin{subequations}
\begin{align}
\cT
&=\theta_1(\phi_1+\phi_2+\phi_3)+ \theta_2(\phi_1+\phi_2(\sigma_1{+}\sigma_2{+}\sigma_3)+ \phi_3),\label{eq:nestingbsp1}\\
c_{\St}
&= \phi_1(\theta_1+\theta_2) +\phi_2(\theta_1+\theta_2(\sigma_1{+}\sigma_2{+}\sigma_3))+ \phi_3(\theta_1+\theta_2)\label{eq:nestingbsp2}
\end{align}
\end{subequations} 
as in Examples~\ref{ex:running0} and~\ref{ex:running2}. These nestings are in one-to-one correspondence with the depicted trees, just as stated in \cref{prop:nesting2}. 
\end{example}

\begin{example}\label{ex:coeff2}
Let $\Theta = \{ \theta_1,\theta_2,\theta_3\}$ and consider the polynomial $f = \theta_1\theta_2+\theta_2\theta_3+ 2\theta_1\theta_3\in\ZZ[\Theta]$. Then $f$ has  nested representation $\theta_1{\cdot}(\theta_2+\theta_3) + \theta_3{\cdot}(\theta_1+\theta_2)$ corresponding to a labeled event tree which is not staged.
\end{example}

\begin{example}\label{example1poly2factor2trees}
Let $\Theta = \{\theta_0, \theta_1, \theta_2, \theta_3, \phi_1,\phi_2\}$ and consider the polynomial
\begin{equation}
f = \theta_0 + 
\theta_1\phi_1 +\theta_1\phi_2 +
\theta_2\phi_1 +\theta_2\phi_2 +
\theta_3\phi_1 +\theta_3\phi_2.
\end{equation}

Then $f$ admits three different nested representations:
\begin{subequations}
\begin{align}
f &= \theta_0 {\cdot}(1)+ 
\theta_1{\cdot}(\phi_1 +\phi_2) +
\theta_2{\cdot}(\phi_1 +\phi_2) +
\theta_3{\cdot}(\phi_1 +\phi_2),\label{eq:examplenest_staged}\\
&= \theta_0 {\cdot}(1)+ 
\phi_1{\cdot}(\theta_1 +\theta_2 + \theta_3) +
\phi_2{\cdot}(\theta_1 +\theta_2 + \theta_3),\label{eq:examplenest_skeleton}\\
&= \theta_0 {\cdot}(1)+ 
\theta_1{\cdot}(\phi_1 +\phi_2) +
\phi_1{\cdot}(\theta_2 + \theta_3) +
\phi_2{\cdot}(\theta_2 + \theta_3).\label{eq:examplenest_notstaged}
\end{align}
\end{subequations}
In particular, \cref{eq:examplenest_staged} corresponds to the staged tree in \cref{fig:notstaged_staged} and \cref{eq:examplenest_skeleton} to the staged tree in \cref{fig:skeleton}. In \cref{sect:polytograph} we show that there are no other staged
trees with interpolating polynomial $f$.  The third nested representation \cref{eq:examplenest_notstaged} corresponds to the labeled event tree in \cref{fig:notstaged} which is not staged.
\end{example}

In the above examples, a given polynomial can admit several different nested representations. By the result below, this is not always the case.

\begin{proposition}[Saturated trees]\label{rem:saturated} 
For a saturated tree $\T$, the interpolating polynomial $\cT$
has a unique nested representation.
\end{proposition}

\begin{proof}
Let $\T'$ be a labeled event tree, not necessarily saturated
    nor staged, with interpolating polynomial $c_{\T'} = c_\T$.     
    We prove that $\T'=\T$, i.e. $\T'$ is indeed the saturated tree~$\T$.   

    Let $C = \supp(c_\T)$
be the set of power-products (or
  \emph{monomials}) in $c_\T$, 
and for a label $x$
 indicate the set of all multiples of that label with
    $C_x=\{t \in c \mid t\text{ multiple of } x\}$.

Let $F=\{\theta_1,\dots,\theta_s\}$ 
and $F'$, respectively,  be the set of root-floret labels of $\T$ and
$\T'$, so $\theta_{v_0}$ in \cref{def:tree} w.r.t.~$\T$
and $\T'$.
We first prove that $F=F'$.
For any $\theta_i\in F$ the power-products 
in $c_{\theta_i}$, corresponding to the root-to-leaf paths originating
from the root-edge in $\T$ which is labeled $\theta_i$, are not multiples of any
$\theta_j$ for $i\ne j$ because $\T$ is saturated.  Thus, if
$F'\subsetneq F$ and $\theta_i\not\in F'$
then the power-products in
$c_{\theta_i}$ could not
correspond to root-to-leaf paths in $\T'$.

It follows that if $F\ne F'$ then there must be a label $\phi\in F'$ with
$\phi\not\in F$. Since $\T$ is saturated,
$\phi$ is the label of only one edge in $\T$, and this edge is, say, 
in the subtree
 starting from the root edge labeled $\theta_1$.
In terms of the power-products,
this implies that $C_\phi \subseteq C_{\theta_1}$.
Hence,  in $\T'$
all root-to-leaf paths originating from the root edge labeled by $\phi$ 
 must have an edge labeled $\theta_1$: see the figure below.
\begin{figure}[h]
\centering
\includegraphics[scale=1]{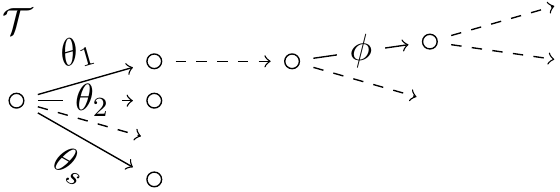}
\qquad
\raisebox{4pt}{\includegraphics[scale=1]{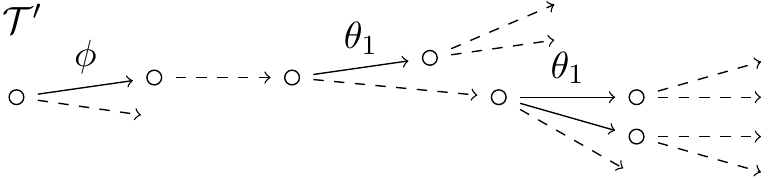}}
\end{figure}

Now consider the root-to-leaf path in $\T'$ where $\theta_1$ appears at
greatest depth, i.e.~with the longest path from the root vertex.
The floret containing $\theta_1$ must have at least another edge so
the paths through this other edge have $\theta_1$ at greater depth. But
this is a contradiction.
Hence ${F=F'}$.  

The subtrees of $\T$ rooted in 
the $s$~children of its root
are again saturated trees, and their
interpolating polynomials are 
$\sum_{t\in c_{\tau}}t$ for $\tau\in\{\theta_1,\ldots,\theta_s\}$ and have 
disjoint sets of labels because $\T$ is saturated.  Therefore we can
repeat the reasoning above on these subtrees and their interpolating
polynomials. We conclude in a finite number of steps that $\T=\T'$.
\end{proof}

Thus when reading an interpolating polynomial from a tree, instead of summing atomic monomials
as in \cref{def:interpol} we can directly use the tree graph to infer a bracketed, nested representation of that polynomial. This representation is in one-to-one correspondence with the labeled graph itself, so the original representation can be easily recovered. Similarly, once we are given \emph{any} polynomial in distributed form and this polynomial admits such a nested bracketing then we can always find a corresponding tree representation.
These insights open the door to replace graphical representations of statistical models by polynomial representations, and hence enable us to employ computer algebra in their study. We will show how this can be done in the next section.


\section{Polynomial and statistical equivalence}\label{sub:equivalence}

Computer algebra is often used to study polynomials that arise naturally in
statistical inference. For instance, context-specific Bayesian networks, staged
trees and chain event graphs are all \emph{parametric} statistical models whose
probability mass function is of monomial form:
$p_{\boldtheta}(x)=\boldtheta^{\alpha_x}
= \theta_1^{\alpha_{x,1}}\cdots \theta_d^{\alpha_{x,d}}$
for every atom $x$ in an underlying sample space where
${\alpha_x=(\alpha_{x,1},\dots,\alpha_{x,d})\in \ZZ_{\ge0}^{d}}$.
This monomial
$\boldtheta^{\alpha_x}$ can then be thought of as for instance a product of
\emph{potentials}~\citep{lauritzen} or simply a product of edge
probabilities in a staged tree with root-to-leaf paths as atoms. So
the network and interpolating polynomials as in \cref{def:interpol} can be
defined for all parametric models admitting a general \emph{monomial
  parametrization} as given above~\citep{manuich2}. 
We can then apply the theory above to these models and employ computer algebra techniques in their study. In particular, often very different parametrizations can give rise to the same model and the interpolating polynomial can help to determine these.

\begin{definition}[Polynomial and statistical equivalence]\label{def:equivalence}
Two staged trees $\T=\pt$ and $\St=\st$ with the same label set $\Theta$ are called \emph{polynomially equivalent} if their interpolating polynomials are equal.

Two staged trees $\T=\pt$ and $\St=\st$ with possibly different label sets, say $\Theta$ and $\Xi$, are called \emph{statistically equivalent} if  
there is a bijection $\Psi:\Lambda(\T)\to\Lambda(\St)$ which identifies their root-to-leaf paths and 
 for any evaluation function on $\Theta$, namely $\Val_{\Theta} :  \Theta \to (0,1)$ extended to $\lambda\in \Lambda(\T)$ as $\Val_\Theta(\lambda) = \prod_{e\in\lambda} \Val_\Theta(\theta(e))$, there exists an evaluation on $\Xi$,  $\Val_\Xi: \Xi \to (0,1)$, such that $\Val_\Theta(\lambda) = \Val_\Xi(\Psi(\lambda))$ for all $\lambda\in \Lambda(\T)$. 
\end{definition}

By definition, two staged trees whose labels are evaluated as probabilities are statistically equivalent if and only if they represent the same statistical model.

Since the interpolating polynomials of polynomially equivalent trees are equal, they are the sum of the same atomic monomials. Therefore there is a bijection between the root-to-leaf paths of polynomially equivalent trees. This implies that polynomially equivalent trees are also statistically equivalent.  
For instance, the trees from Examples~\ref{ex:running2} and \ref{ex:running1} are polynomially, and so statistically equivalent. In particular, the interpolating polynomial is sufficient to determine a probability distribution up to a permutation of the values it takes across an underlying sample space. 
 
From \cref{prop:nesting2}, the class of polynomially equivalent trees is fully described by \emph{all} nested representations of the interpolating polynomial. Indeed, when reordering the terms of a nested representation as in \cref{fig:nesting}, the atomic monomials of the underlying tree do not change. So if we are given the interpolating polynomial of a 
staged tree and we can find all its possible nested representations then we have automatically found all of its polynomially equivalent tree representations -- and often a large subclass of the whole statistical equivalence class. For example, in the case of decomposable Bayesian networks the equivalence class of a polynomial given in clique parametrization contains the Markov-equivalence class~\citep{CGphdthesis}.

Polynomially equivalent trees can be thought of as those having the same parametrization. However this parametrization is often read in a different \emph{non-commutative} way for different graphical representation in that class. For instance, the staged trees in Examples~\ref{ex:running2} and~\ref{ex:running1} have the same atomic monomials belonging to identified atoms but $\pithetalambda=\theta_1\phi_1$ in $\pt$ and $\pi_{\theta'}(\lambda')=\phi_1\theta_1$ in $\st$ for identified atoms $\lambda$ and $\lambda'$. Analogous instances of this phenomenon occur in the class of decomposable Bayesian networks where a model parametrization can be given by potentials on cliques which are renormalized across different graphical representations of the same model.

Statistically equivalent trees however can be thought of as reparametrizations of each other, very much like in Bayesian networks where a parametrization can either be based on parent relations between single nodes in a graph or alternatively on clique margins.
See also \cref{toricExInd}.

\begin{example}\label{bsp:skeleton} 
Polynomially equivalent trees can often be described by a variety of different graphs. For instance, the polynomial $c=\theta_0+(\theta_1+\theta_2+\theta_3)(\phi_1+\phi_2)$ has at least three different labeled trees associated: see \cref{fig:labeledevent} and \cref{example1poly2factor2trees}.

The two trees in Figs.~\ref{fig:notstaged_staged} and~\ref{fig:skeleton} are polynomially equivalent representations of the same model on seven atoms. The tree in \cref{fig:notstaged} is not because it is not a staged tree. In particular, this tree is not a probability tree because sum-to-$1$ conditions imposed on its florets would be contradictory.
\end{example} 

\begin{example}[Maximal representations]\label{rem:binary}
For any labeled event tree there exists a statistical equivalent \emph{binary} labeled event tree whose graph $\baretree$ is such that $\#\floretfun{v}\in\{0,2\}$ for all $v\in V$. This can be thought of as a \emph{maximal} representation within the class of statistically equivalent trees. 
We can easily obtain a binary tree by splitting up each floret with strictly more than two edges as shown in \cref{fig:maximal_representation}. 
\begin{figure}[t]
\centering 
\includegraphics[scale=1]{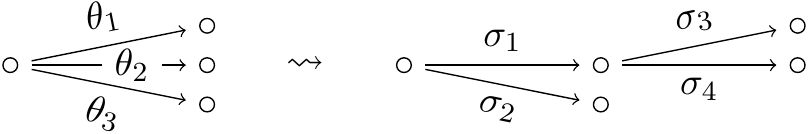}
\caption{Maximal and minimal representations of a floret. See Examples~\ref{rem:binary} and~\ref{example:minrepre}.}
\label{fig:maximal_representation}
\end{figure}
In particular, for a floret in a probability tree labeled by $\theta_1,\theta_2,\theta_3$, we would obtain new labels $\sigma_1,\sigma_2,\sigma_3,\sigma_4$ which are renormalizations of the original parameters such that sum-to-$1$ conditions hold, $\sigma_1+\sigma_2=1$ and $\sigma_3+\sigma_3=1$, while retaining the distribution over the three depicted atoms, so $\sigma_1=\theta_1+\theta_2$, $\sigma_1=\nicefrac{\theta_1}{\theta_1+\theta_2}$, $\sigma_2=\nicefrac{\theta_2}{\theta_1+\theta_2}$  and $\sigma_2=1-\sigma_1$. 
\end{example}

\begin{example}[Minimal representations]\label{example:minrepre}
In the polynomial equivalence class of a saturated tree there is exactly one member, namely the tree itself. This is because, by \cref{rem:saturated}, for saturated trees the nested representation of an interpolating polynomial is unique. The statistical equivalence class of a saturated tree however is much bigger. This is a consequence of \cref{rem:binary} above.  
In particular, for every saturated tree there is a unique \emph{minimal} graphical representation given by a single floret whose labels are the atomic monomials (or joint probabilities) and whose number of edges coincides with the number of root-to-leaf paths in any equivalent representation.
\end{example}

In the development in this paper we mainly focus on a parametric
characterization of staged tree and other statistical models. This naturally links in with an alternative \emph{implicit}
characterization which is well known in algebraic statistics. For
instance, a polynomial representation of a Bayesian network involving
exclusively the joint probabilities -- i.e.~the values of the associated
probability mass function $p(x)$ as $x$ varies in the sample space -- can be derived from the
equalities $p(x)=\boldtheta^{\alpha_x}$ using ring operations. The
algebraic theory behind this is called  \emph{elimination theory}~\citep{KR:book1} 
of which Gaussian elimination for solving systems of linear equations is a simple example. 
The representation of a Bayesian network as such a set of polynomials is an algebraic structure called a \emph{toric ideal} and has great importance in algebraic statistics: see e.g.~\cite{pistonerw:algstat,geigermeeksturm:algebra,oberwolfach}.

Notably, this alternative characterization can also be used to describe statistical equivalence -- though in a less constructive way than the method we present here and without immediate links to a graphical representation of a model.

\begin{figure}[t]
\centering
\subfloat[A staged tree representing a binary independence model.]{
\centering
\includegraphics[scale=1]{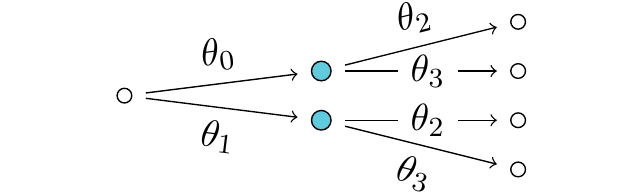}
\label{fig:toricex1}}
\quad
\subfloat[Minimal representation of the
saturated model on four atoms.]{
\centering
\includegraphics[scale=1]{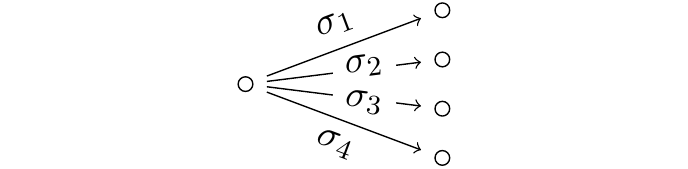}
\label{fig:toricex2}}
\caption{Trees analyzed in \cref{toricExInd}.}
\end{figure}

\begin{example}\label{toricExInd}
The labeled event tree in \cref{fig:toricex1} is a staged tree on four atoms with labels 
$\Theta=\{\theta_0, \theta_1, \theta_2, \theta_3\}$. 
The equalities holding for the four atomic monomials
\begin{equation}
p_1= \theta_0\theta_2 , \, \, p_2= \theta_0\theta_3 , \,\, p_3= \theta_1\theta_2 , p_4= \theta_1\theta_3
\end{equation}
imply the equality  $p_1p_4 = p_2p_3$. 
This parametrization of the model in \cref{fig:toricex1} is not to be confused with the minimal representation of the saturated model on four atoms in \cref{fig:toricex2}.    

An interpretation of this equation is as follows. Assume two binary random variables $X,{Y\in \{0,1\}}$ are such that
\begin{align}
\operatorname{Pr}(Y=1, \, X=1 ) &= p_1, \qquad \operatorname{Pr}(Y=0, \, X=1 ) = p_2, \\
\operatorname{Pr}(Y=1, \, X=0 ) &= p_3, \qquad \operatorname{Pr}(Y=0, \, X=0 ) = p_4 .
\end{align}   
Then $p_1p_4 = p_2p_3$ is an instance of a fundamental relationship in algebraic statistics for representing conditional independence of discrete random variables: see e.g.\ \cite[Section 6.10]{pistonerw:algstat} and \cite[Proposition 3.1.4]{oberwolfach}.
In this specific case the equality implies that $X$ and $Y$ are independent. 
%
%
\end{example}

\section{From polynomials to trees: finding the nested
  representations}\label{sect:polytograph}


\subsection{Potential root-floret  labels and square-free monomials}

Building on the results above we can now use methods from commutative
algebra to compute all the staged trees with
a given interpolating polynomial and so to compute a complete polynomial equivalence class.
The two key notions we use to build an algorithm which determines these classes are those of a monomial ideal and of its primary decomposition which, for square-free monomials, coincides with the prime decomposition. These notions are recalled in the appendix.

The key of the proposed algorithm is \cref{thm:minimal} below.  This states in algebraic terms  that
for any tree $\T=(T, \theta)$  each monomial in $c_\T$ is divisible by some label in the set ${F}=\theta_{v_0}$ of the floret labels belonging to the root of $\T$, and that ${F}$ is minimal (with respect to inclusion) with this property.

\begin{theorem}\label{thm:minimal}
Let $\T$ be a staged tree.
The monomial ideal  $\ideal{\theta_{v_0}}$ generated by the root-floret labels
is a minimal prime of the ideal $\ideal{\supp(\cT)}$ generated by the support of $c_\T$.
\end{theorem}

\begin{proof}
Let ${F}=\theta_{v_0} = \{\theta_1, ..., \theta_s\}$ be the set of root-floret labels. 
Then each power-product in $c_\T$ is a multiple of some label in~$F$.
Because it is generated by indeterminates, $\ideal{F}$ is a prime ideal 
 containing all power-products in $c_\T$. 
Suppose, by contradiction, that ${F}$ is not minimal.
Then there exists $\tilde{F}\subsetneq F$
with  $\ideal{\tilde{F}}$ containing all power-products in $c_\T$.
Without loss of generality let $\tilde{F}=\{\theta_2, ... , \theta_s\}$.
Now, each root-to-leaf path starting with the root edge labeled
$\theta_1$ has an associated atomic monomial
$\theta_1t\in\supp(c_\T)\subseteq\tilde{F}$, $j\geq 2$.
 Therefore $\theta_1t= \theta_1\theta_jt'$ for some $\theta_j\in\tilde{F}$.
As $\T$ is staged, this implies that the whole root floret~$F$ must appear
again in the subtree: see the illustration below.
\begin{center}
\includegraphics[scale=1]{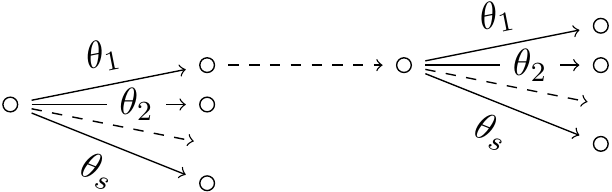}
\end{center}

Next consider the subtree containing the repeated root-floret labels at a
minimum depth and repeat the reasoning above: 
each root-to-leaf path containing the two edges labeled
$\theta_1$ corresponds to an atom $\theta_1^2t\in\supp(c_\T)$ and is 
therefore a multiple of some label in $\tilde{F}$.
Then the whole root floret is repeated again deeper in the subtree,
producing some atom divisible by $\theta_1^3$.
Since this reasoning can be repeated a finite number of times, we 
have the contradiction that there is an atomic monomial divisible by a power of~$\theta_1$ and
by no label in~$\tilde{F}$. 
Therefore  ${F}=\{\theta_1, ..., \theta_s\}$  is minimal.
\end{proof}

 \begin{example} \label{primdecperex}
The interpolating polynomial $\cT$ in \cref{ex:running0} has support 
\begin{equation}
\supp(\cT)=\{
\theta_1\phi_1,\,\, \theta_1\phi_2 ,\,\, \theta_1\phi_3 
,\,\, \theta_2\phi_1 
,\,\, \theta_2\phi_2\sigma_1,\,\, \theta_2\phi_2\sigma_2 ,\,\, \theta_2\phi_2\sigma_3 
,\,\, \theta_2\phi_3   
\}.
\end{equation} 
The primary decomposition of the corresponding square-free monomial ideal is 
\begin{align}
\ideal{\supp(\cT)} & = 
\ideal{\phi_1,\phi_2,\phi_3} \cap \ideal{\theta_1,\theta_2} \cap \ideal{\phi_1,\phi_3,\theta_1,\sigma_1,\sigma_2,\sigma_3}.
\end{align}
Therefore, by \cref{thm:minimal}, there are three different sets of possible root labels
for a staged tree with
interpolating polynomial $\cT$. 
We show in \cref{ex_alg} below that the polynomial equivalence class of $\cT$ is given by just two trees. 
\end{example} 

\begin{example}\label{ex:non-square-free}
Consider the interpolating polynomial ${\cT=\theta_0+\theta_1\phi_1+\theta_1\phi_0\phi_1+\theta_1\phi_0^2}$. 
The minimal prime decomposition  of $\ideal{\supp(\cT)}$ is given by two sets, namely
$\ideal{ \theta_0,\theta_1 }$ and $\ideal{ \phi_0, \theta_0, \phi_1}$. 
The first one leads to the tree in \cref{stab:non-square-free}. It can be shown by exhaustive search that the second does not give the labels of a root floret in a labeled event tree. 

\begin{figure}[t]
\centering
\includegraphics[scale=1]{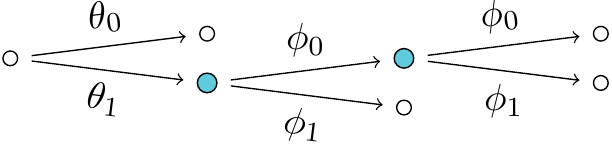}
\caption{A tree whose interpolating polynomial has a non-square-free
  atomic monomial. See \cref{ex:non-square-free}.} 
\label{stab:non-square-free}
\end{figure} 
\end{example}

The key assumption in \cref{thm:minimal} is that the input tree $\T$ is staged, otherwise the result need not be true.

This theorem is central to the algorithm we present in the following section because it shows that
instead of searching for root-floret labels among all subsets of labels $\Theta$, 
the search can be limited 
to those subsets which are the generators of the minimal primes
of $\ideal{\supp(\cT)}$. 
If $\Theta$ has $d$ elements,
their number is bounded above by
$d\choose{\lceil{d/2}\rceil}$
whereas the number of the subsets of $\Theta$ is $2^d$.
 So considering all possible subsets
of $\Theta$, and having to repeat this recursively, may lead to a
combinatorial explosion of cases to analyze. As a consequence, \cref{thm:minimal} gives a drastic reduction of the set of candidate root-floret labels. 
\medskip

Staged trees whose interpolating polynomials are sums of square-free power-products are interesting cases both from an algebraic viewpoint and for their interpretation in statistical inference. For instance, if all power-products in $c_\T$ are square-free then 
the proof of \cref{thm:minimal} can be shortened obtaining the contradiction by \cref{prop:rootfloretlabels} directly.
In terms of staged tree models, this condition implies that if a unit passes through a vertex in a given stage it cannot subsequently pass through another vertex in the same stage. By making this requirement we can avoid various complex ambiguities associated with exactly how we relate a sample distribution to a polynomial family. Although less useful in modeling time series, in most cross-sectional statistical models this constraint will almost always apply.

The restriction to polynomials with square-free support enables us to prove the second and third central result for our algorithmic implementation.

\begin{proposition}[Root-floret labels]\label{prop:rootfloretlabels}
Let $\T$ be a staged tree
whose interpolating polynomial $\cT=\sum_{(v_0,w)\in \floretfun{v_0}} \theta(v_0,w) \cdot c_{\T_w}$ is a sum of square-free power-products.
Then no label in $\theta_{v_0}$ appears in any subtree-polynomial $c_{\T_w}$.
\end{proposition}

\begin{proof} 
Because $\T$ is a staged tree we have $\theta_{v_0} \cap
\theta_v=\emptyset$ or $\theta_{v_0} = \theta_v$ for all $v\in
V{\setminus} \{ v_0\}$ by \cref{def:interpol}. 
By contradiction, suppose there is a subtree $\T_w$ containing 
a floret with labels $\theta_{v_0}$.
Let $\theta_1$ be the label of the edge $(v_0,w)$ for some $w\in V$. Then there is a root-to-leaf path with at least two edges labeled $\theta_1$: see also the illustration in the proof of \cref{thm:minimal}.
Hence there is a multiple of
$\theta_1^2$ in $\cT$. This is a contradiction because $\cT$ is a sum of square-free
power-products.
So there is no subtree~$\T_w$ containing 
a floret with labels~$\theta_{v_0}$.
The claim follows from \cref{cor:rootexpansion}. 
\end{proof} 

\begin{corollary}\label{cor:coeff1}
Let $\T$ be a staged tree
whose interpolating polynomial $c_\T$ is a sum of square-free
power-products.
Then all coefficients in $c_\T$ are equal to~1.
\end{corollary}

\begin{proof}
The claim follows from \cref{prop:rootfloretlabels} and its 
recursive application to subtrees of~$\T$.
\end{proof}

So when searching for staged trees using square-free interpolating polynomials,
coefficients might be ignored. This is not true for labeled event trees by \cref{ex:coeff2}.
In \cref{sub:discussion} we will see that this result will allow the application of the algorithm in \cref{sect:algorithm} to network polynomials of staged trees.

\FloatBarrier
\subsection{The algorithm \texttt{StagedTrees}}\label{sect:algorithm}

Given a polynomial $f$ whose power-products are square-free and
with coefficients all equal to one, 
there is an obvious algorithm which determines all its nested representations, and 
in particular all staged trees for which $f$ is the interpolating polynomial. This algorithm is here called \texttt{StagedTrees} and is given in pseudo-code in \cref{alg:factorisation}.
Following the notation in \cref{def:nestedpoly}, 
the proposed algorithm searches over subsets $A\subseteq \Theta$ of the indeterminates
appearing in $f$ and recursively checks whether it is possible to construct the polynomials
$f_x$ for $x\in A$. The choices of $A$ are hereby constrained to the minimal primes of the monomial ideal associated to $f$ as determined by \cref{thm:minimal}.
This algorithm works even when it is not known \textit{a priori} whether or not $f$ is the interpolating polynomial of a staged tree. Since the support of $f$ is finite it is clear that the recursion terminates.
The function \texttt{StagedTrees}
is part of the \texttt{CoCoA} distribution from version~5.1.6
({\small
\texttt{http://cocoa.dima.unige.it/download/}\linebreak\texttt{CoCoAManual/html/cmdStagedTrees.html}}).
\medskip
 
\begin{algorithm}
\SetKwInOut{Input}{Input}\SetKwInOut{Output}{Output}
\Input{$C=\supp(f)$ a set of square-free power-products over  a set of indeterminates $\Theta$ for a polynomial $f=\sum_{t\in C}t\in\ZZ[\Theta]$ with all coefficients zero or one.}
\Output{The set $W$ of all staged-trees with interpolating polynomial $f$.}
Let  $W = \emptyset$ (initialise the output set of trees)\\
{\If{$C = \{1\}$}{\nonl \textbf{return} a single-vertex tree}}
{\If{$\#C=1$ has only one element}{\nonl \textbf{return} the emptyset $\emptyset$}}
{\If{$C\subseteq\Theta$ is a subset of indeterminates and $\#C\geq2$ has at least two elements}{\nonl \textbf{return} the staged tree made of the single floret labeled by $\Theta$}}
\nonl\Else{compute the prime decomposition  $\{F_1,\dots,F_k\}$ of the square-free monomial ideal $\ideal{C}$\\
\For{each $i=1,\ldots,k$}{\nonl consider $F_i$ and proceed as follows:\\
\setcounter{AlgoLine}{0}
\SetNlSty{textbf}{6.}{}
\For{each indeterminate $x\in F_i$}{\nonl define $C_x = \{ t \in C \mid t \text{ is a multiple of }x\}$}
\If{there exist $y\ne x$ such that $C_x\cap C_y\ne\emptyset$}{\nonl discard $F_i$ and go to next minimal prime in Step~6}
\For{each indeterminate $x\in F_i$}{
\setcounter{AlgoLine}{0}
\SetNlSty{textbf}{6.3.}{}
\If{$\#C_x=1$ has only one element and this is is not equal to $x$, so $C_x\ne\{x\}$}{\nonl discard $F_i$ and go to next minimal prime in Step~6}
define a set $C'_x = \{\frac{t}{x} \mid t \in C_x\}$ of square-free power-products over $\Theta{\setminus}\{F_i\}$;\\
call \texttt{StagedTrees} recursively with input $C'_x$ and obtain the set $W_x$ of all staged trees with interpolating polynomial $\sum_{t\in C'_x}t$;\\
\If{$W_x=\emptyset$ is the emptyset}{\nonl discard $F_i$ and go to next minimal prime in Step~6}
}
\setcounter{AlgoLine}{3}
\SetNlSty{textbf}{6.}{}
construct the set $W'$ of all trees with root-floret labels $F_i = \{x_1,\dots,x_{r_i}\}$ and\\
     \nonl construct the subtrees $(T_{x_1},\dots,T_{x_{r_i}})\in W_{x_1}\times\dots\times W_{x_{r_i}}$
     where each $T_{x_j}$ is rooted at the second vertex of the edge labeled $x_j$;\\
discard from $W'$ all trees which are not staged;\\
redefine $W$ as $W \cup W'$;
}}
\caption{\texttt{StagedTrees}: Inferring all nested representations of a given polynomial.}\label{alg:factorisation}
\end{algorithm} 
 
The base steps of the recursion in \cref{alg:factorisation} are given by the simplest trees:
a single vertex tree for $C={1}$ (Step~2), 
or a floret without subtrees for $C\subseteq \Theta$ (Step~4)
with at least two edges (Step~3).
Compare also the recursive description in \cref{prop:nesting1}. 
In Step~5, \cref{thm:minimal} is applied to determine 
the candidate root-florets $F_1,\dots,F_k$.
The main loop in Step~6
considers each $F_i$ one at a time, 
and determines all the staged trees
having root floret $F_i$, $i=1,\ldots,k$.

In the main loop, Step~6.2 checks if the subsets defined in Step~6.1 
give a partition for $C$ which is a necessary condition from 
 \cref{prop:rootfloretlabels}:
since $F_i$ is a minimal prime for $\ideal{C}$
it follows that $C = \cup_{x\in F_i}C_x$. Therefore 
only disjointness needs to be verified.
Then the inner loop in Step~6.3, with its sub-steps, 
considers one at a time each $x\in F_i$, 
and determines (if possible) all the subtrees emanating from the
second vertex of the edge labeled $x$.
In particular, Step~6.3.1 stops the search for $F_i$ if there is a single
emanating edge and therefore by definition not an event tree.
Step~6.3.3 makes the recursive call on $C'_x$
(defined in Step~6.3.2) to determine the set $W_x$ of all possible
subtrees from~$x$. 
If $W_x$ is empty then Step~6.3.4  stops the search for~$F_i$.

Concluding the main loop, Step~6.4 is reached if for each edge having a label 
in $F_i$ there is at least one subtree. Then the floret labeled by $F_i$ together with all
combinations of its subtrees make a set $W'$ of event trees,
with root-floret labels $F_i$, whose interpolating polynomial is the sum of the monomials in $C$. 
At this point Step~6.5 discards those which are not staged. 
In particular the subtrees are staged, and compatibility of stages across the
subtrees is checked here, in the obvious way. Finally, Step~6.6 stores them in $W$.

\begin{figure}[t]
\centering
\subfloat[$F_1$ leads to a non-staged tree.]{
\includegraphics[scale=1]{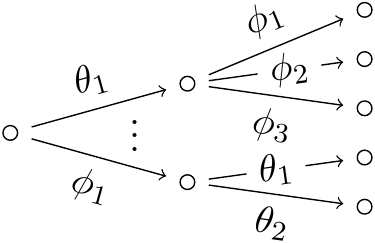} 
\label{f1_ex_alg}}
~
\subfloat[$F_2$: subtrees from $\theta_0$ and $\phi_2$.]{
\includegraphics[scale=1]{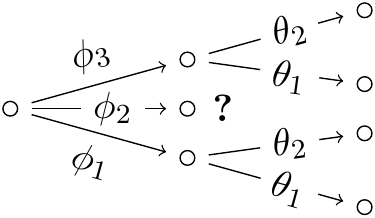}
\label{f2_ex_alg}}
~
\subfloat[$F_2$: subtree from $\phi_1$ is not an event tree.]{
\includegraphics[scale=1]{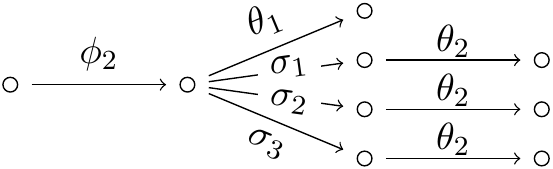}
\label{f3_ex_alg}}
\caption{The working of the {\tt StagedTrees} algorithm.
See \cref{ex_alg}.}
\end{figure}

\begin{example} \label{ex_alg}
We illustrate the working of the \texttt{StagedTrees} algorithm on \cref{ex:running0}.
From \cref{primdecperex} we can consider only three sets of potential root-floret labels of staged trees with interpolating polynomial $\cT$ given in  \cref{eq:exatomic}. These are:
\begin{align} 
F_1 & =\{ \phi_1,\phi_3,\theta_1,\sigma_1,\sigma_2,\sigma_3  \} \\
F_2 & =\{ \phi_1,\phi_2,\phi_3 \} \\
F_3 & =\{ \theta_1, \theta_2 \}  .  
\end{align}
The first set  $F_1$ cannot be a floret-label set 
because $C_{\phi_1}  \cap C_{\theta_1} \ne \emptyset$, 
 see Step 6.2 in the algorithm. 
Indeed the two sets 
\begin{equation}
\begin{array}{lll}
C_{\phi_1} &= \{ \theta_1 \phi_1, \theta_2 \phi_1 \} & = \phi_1 \{  \theta_1, \theta_2 \} \\ 
C_{\theta_1} &= \{ \theta_1 \phi_1, \theta_1 \phi_2, \theta_1 \phi_3 \} & = \theta_1 \{ \phi_1, \phi_2, \phi_3  \} 
\end{array}
\end{equation}
show that, if $F_1$ were a floret-label set, then the tree would include a structure such as in \cref{f1_ex_alg} 
which cannot be part of a staged tree: see also \cref{cor:rootexpansion} and \cref{prop:rootfloretlabels}. Above we have used the convention that the product of a single label with a set of labels is defined as the set of all elementwise products.

With $F_2$ in the first step of the algorithm we have 
\begin{equation}
\begin{array}{lll}
C_{\phi_3} &= \{\phi_3 \theta_1, \phi_3  \theta_2 \} &= \phi_3  \{\theta_1, \theta_2\} \\ 
C_{\phi_2} &= \{ \theta_1 \phi_2, \theta_2 \phi_2 \sigma_1,
             \theta_2 \phi_2 \sigma_2, \theta_2 \phi_2 \sigma_3\}  
                  &= \phi_2 \{ \theta_1 , \theta_2  \sigma_1, 
                  \theta_2 \sigma_2, \theta_2  \sigma_3  \}  \\
C_{\phi_1} &= \{ \theta_1 \phi_1, \theta_2 \phi_1 \} &= \phi_1 \{  \theta_1, \theta_2 \}
\end{array}
\end{equation}
The algorithm calls recursively
on the sets $C'_{\phi_3}$ and $C'_{\phi_1}$ but
stops immediately (Step~4 in the algorithm)
as summarized in \cref{f2_ex_alg}. 
For the middle branch we need to continue the recursion by working on 
$C'_{\phi_2} $.
The monomial ideal generated by $C'_{\phi_2} $
has the following primary decomposition   
\begin{equation}
\ideal{C'_{\phi_2}}  
=\ideal{\theta_1,\theta_2} \cap \ideal{ \theta_1 ,  \sigma_1,\sigma_2,\sigma_3 }  .  
\end{equation}
Taking $F=\{ \theta_1,\theta_2 \}$ gives the tree in \cref{fig:nestingb}
while $F=\{ \theta_1 , \sigma_1,\sigma_2,\sigma_3 \}$ leads to the
situation in \cref{f3_ex_alg} which does not correspond to an
event tree. 
In conclusion, $F_2$ gives the tree in \cref{fig:nestingb} only. 
The result of the algorithm starting from $F_3$ is analogous and leads to the tree in \cref{fig:nestinga}. 
\end{example}

\FloatBarrier
\subsection{Discussion of the algorithm}\label{sub:discussion}

It was shown in~\cite{meins} that the application of two graphical operators called the \enquote{swap} and \enquote{resize} on a staged tree could be used to traverse a statistical equivalence class. However these authors did not provide an implementation of their graphical methods in algebraic or computational terms.
So \cref{alg:factorisation} fills that gap and enables us to determine the full polynomial equivalence class of a given staged tree. We hereby focus on staged as opposed to labeled event trees because these can always be interpreted as representations of statistical models as in \cref{sect:graphtopoly,sub:equivalence}. Of course our new algorithm can be easily adapted to discover more general representations. We will now discuss some of the properties of this algorithm.
\medskip

First, the \texttt{StagedTrees} algorithm can be modified to work on non-square-free power-products.
For this purpose Step~6.2 must be disabled and all the possible partitions
of $C$ need to be checked, making the algorithm more expensive.
For example, the only minimal prime for the ideal ${\ideal{\supp(\theta_1 + \theta_2  \cdot(\theta_1+\theta_2))}}$ is ${\ideal{\theta_1,\theta_2}}$ which leads to two partitions $\{\theta_1, \theta_1\theta_2\}, \{\theta_2^2\}$, and $\{\theta_1\}, \{\theta_1\theta_2,\theta_2^2\}$.  
Calling the algorithm on the first partition gives no answer because
it leads to a tree which is not an event
tree, whereas the second gives the original nested representation.
Moreover, in this partitioning one also needs to keep track of the
coefficients: as illustrated by the nested representation
$\theta_1 \cdot(\theta_1+\theta_2) + \theta_2  \cdot(\theta_1+\theta_2) = \theta_1^2 + \theta_1\theta_2 + \theta_2^2$.

Second, so far we often emphasized the use of the \emph{interpolating polynomial}
as opposed to the \emph{network polynomial} in
\cref{def:interpol}.
 This was to highlight the \emph{structure}
of the tree, as opposed to the real values associated its root-to-leaf
paths: compare also \cref{def:equivalence}.
However, if $c_{g,\T}$ is the network polynomial associated to a
staged tree $\T$ and its \textit{power-products are square-free}, from
\cref{prop:rootfloretlabels} it follows that the root-to-leaf paths
$\lambda\in \Lambda(\T)$ are labeled by distinct
  monomials.  This means that in the network
polynomial the coefficients $g(\lambda)$ are kept distinct.  In
conclusion, all staged trees with a given network polynomial
$c_{g,\T}$ are found by the algorithm \texttt{StagedTrees} applied to $C =
\supp(c_{g,\T})$.
Afterwards the coefficients $g(\lambda)$ 
can be
associated to the corresponding root-to-leaf paths. 

Third, thanks to the reduction to minimal primes, the algorithm is very fast also for real-world settings. In \cref{sect:CHDS} we will apply \texttt{StagedTrees} to discover the polynomial equivalence class of a staged tree describing a real problem with 24 atomic events. This computation takes much less than a second on a laptop with a 2.4~GHz Intel Core~2 Duo processor. Similarly, it takes 2.3 seconds to compute the 576 staged trees sharing the interpolating polynomial ${(\theta_0+\theta_1)(\phi_1+\phi_2)(\tau_0+\tau_1)(\sigma_0+\sigma_1)}$ representing four independent binary random variables: compare \cref{fig:toricex1}. Computing the polynomial equivalence class of four independent random variables taking three levels each takes significantly longer at 12:23min but produces 55,296 different staged trees, each having 81 atoms. Naturally, the more stage structure there is present the more different polynomially equivalent representations are possible, so the latter two are somewhat extreme cases. On medium-sized real-world applications like the one presented below our computations are very fast. So this algorithm allows us to systematically enumerate and analyze staged trees of the same order or even bigger than the study we will consider.

Fourth, every Bayesian network, context-specific Bayesian network~\citep{boutilier} and object-oriented Bayesian network~\citep{OOBN} can be represented by a staged tree where inner vertices correspond to conditional random variables and the emanating edges correspond to the different states of these variables. Then two vertices are in the same stage if and only if the corresponding  rows of conditional probability tables are identified. 
For instance, the independence model of two binary random variables can be represented by the staged tree depicted in \cref{fig:toricex1}. The complete Bayesian network on two binary random variables can be represented by the staged tree in \cref{fig:toricex2}.
However, staged trees allow for much less symmetric -- and hence more general -- modeling assumptions. In particular, they do not rely on an underlying product-space structure but can express relationships directly in terms of events. So this class of models is much larger than the Bayesian network class and as a consequence the \texttt{StagedTrees} algorithm can be optimized to traverse this wider class as well as the class of Bayesian networks.

So the methodology we developed for the \texttt{StagedTrees} algorithm will serve as a springboard for really fast algorithms to analyze equivalence classes of staged trees and in the future causal discovery algorithms over this class: see also \cref{sect:CHDS}. We illustrate below that these computer algebra analysis enable us to obtain further insights about the properties of the underlying class of statistical models.

\FloatBarrier
\section{Additional properties of interpolating polynomials}
\label{sect:treecomp}

A natural question to ask is whether or not a given polynomial can be
seen to be the interpolating polynomial of an event tree
\emph{without} having to construct a nested representation first.  The
following proposition gives some necessary conditions for a polynomial
to be an interpolating polynomial of a labeled event tree.

Recall that for a power-product 
$\boldtheta^a=\theta_1^{a_1},\ldots,\theta_d^{a_d}$, the degree is the
sum of the exponents, $\deg(\boldtheta^a)=\sum_{j=1}^da_j$, 
and for a polynomial $c= \sum_{i=1}^d \boldtheta^{\alpha_i}$
the degree is $\deg(c) = \max\{\deg(\boldtheta^{\alpha_i})\}$.

\begin{proposition}\label{prop:characterisation}
Let $c(\boldtheta)=\sum_{i=1}^n\boldtheta^{\alpha_i}$ be a  polynomial
with square-free support, 
i.e.~$\alpha_i=(a_{i1},\dots,a_{id})\in\{0,1\}^d$ for all $i=1,\ldots,n$ and some $d\ge1$. 
If there exists a labeled event tree such that $c$ is its interpolating
polynomial then the following conditions hold:
\begin{enumerate}
\item If $c\ne1$ then $d, n\geq 2$ and $d\leq 2n-2$, and $d$ 
$>\deg(c)$.

\item The frequency with which each root label appears in the monomials $\boldtheta^{\alpha_i}$, $i=1,\ldots,n$, is greater than the degree of the monomials in which they appear.

\item If the degree of
$\boldtheta^{\alpha_i}$ is equal to the degree
 of $c$, then there exists $\boldtheta_j$ with $i\not=j$ with the same degree as $\boldtheta^{\alpha_i}$ and the degree of the greatest common divisor of $\boldtheta_j$ and $\boldtheta_i$ is equal to the degree of $c$ minus one.
\item No power-product in the support of $c$ can be a proper multiple of
another.
\end{enumerate}
\end{proposition}

\begin{proof}
\begin{enumerate}
\item
The root floret of a labeled event tree with at least one edge has at
least two edges with distinct labels, thus $d, n\geq 2$.  
We prove the claim by induction on the number of florets in a labeled
event tree. 
Let $E$ be the set of edges and $L$ the set of leaves of the tree.
If a tree is formed by a single vertex then $\#E=0$ and $\#L=1$.
Therefore $\#E=0 = 2\#L{-}2$.  By induction suppose that 
$\#E \le 2\#L-2$ for the tree $\T$. Consider the tree $\T'$ obtained by adding to a leaf in $\T$ a
floret with $s$ edges. Because $s\ge 2$, thus $s\le 2s{-}2$ and hence  
$\#E'=\#E{+}s$ and $\#L'=\#L {+} s{-}1$.
As a result,  $\#E' = (\#E)+(s) \le (2\#L {-}2)+(2s{-}2) = 2(\#L+s{-}1) -2 = 2\#L'{-}2$.
We conclude by noticing that $d\le \#E$ and $n=\#L$.

\item
 Consider \cref{fig:propbound}. In labeled event trees, an atomic monomial of degree $l\in\mathbb{N}$ is associated to a root-to-leaf path of length $l$. This path has one bifurcation at every vertex, so is embedded in a graph with at least $l+1$  distinct root-to-leaf paths. So every root-label $\theta_1$ occurs in monomials of maximal degree $l$ and there are at least $l+1$ of those.

\item
 Because $\#\floretfun{v}\geq 2$ for all $v\in V$, every leaf-floret has two edges. There are hence at least two monomials of the same maximal degree, namely those belonging to the longest paths in the tree: these are equal until they split at a leaf-floret.
\begin{figure}[t]
\centering
\includegraphics[scale=1]{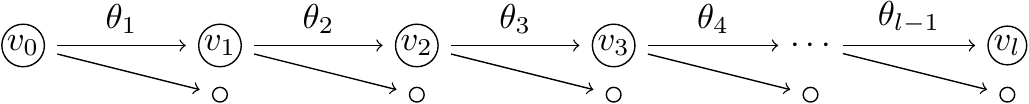}
\caption{A root-to-leaf path $\lambda=(e_1,\ldots,e_l)$ in an event tree.}\label{fig:propbound}
\end{figure}

\item
  Let $t_1$ and $t_2$ in $c$ be multiples of each other, written as
  $t_1|t_2$.  They are atomic monomials of two root-to-leaf paths,
  $\lambda_1$ and $\lambda_2$, which are not empty if~$\T$ is not
  trivial. Let $e$ be the root edge
 labeled $\theta_1$, the first edge in $\lambda_1$.
Then $\lambda_2$ starts with the same edge: otherwise
 $\theta_1| t_1$,  and $\theta_1\not|\, t_2$
 for \cref{prop:rootfloretlabels}.
Therefore 
we can repeat the reasoning on $\lambda_1{{\setminus}}\{e\}$ and
$\lambda_2{{\setminus}}\{e\}$ in the subtree $\T(w)$. 
After a finite number of steps we can then conclude $\lambda_1=\lambda_2$ and
thus $t_1=t_2$.
\end{enumerate}
\end{proof}

The conditions in \cref{prop:characterisation} are necessary but not sufficient.

\begin{example} \label{ex:propertiesarenotsufficient}
The polynomial $\theta_1 \phi_1+\theta_1 \phi_2+\theta_2 \theta_3 \theta_4+\theta_2 \theta_3 \phi_1+\theta_2 \theta_4 \phi_2$ satisfies all points  in \cref{prop:characterisation}. However, it cannot be written in the form of a nested representation. It is thus not the interpolating polynomial of a labeled event tree.
\end{example} 

\FloatBarrier
\section{Two other representations of labeled event trees}

From the previous section we see that if there is a labeled event tree for a square-free polynomial $c$ with $n$ terms then that tree has $n$ root-to-leaf paths. 
Every such path is labeled by a monomial $\boldtheta^\alpha$ which is a power-product in $\supp(c)$. 
We next present two well-known alternative representations of these atomic monomials of a staged tree. 

The first representation is based  on the notion of an abstract simplicial complex, i.e.\,a family $\G$ of subsets of a finite set (the nodes of the simplicial complex) such that if $A\in \G$ and $B\subseteq A$ then $B\in \G$.
In our case the nodes of the simplicial complex are the labels $\Theta$ of a labeled event tree $\T=\pt$ and the family is given by the monomials $\pi_\theta(\lambda_i)=\boldtheta^{\alpha_i}$, $i=1,\ldots,n$, and all of their divisors. For an illustration see \cref{fig:simplicialcomplex}.
This graphical representation for a set of monomials has been successfully used in the data analysis of complex systems~\citep{Carlsson2009,PirinoetAl2015,DonatoetAl2016}.

\begin{proposition} A labeled event tree 
$\T$ is saturated with  root labels $\theta_1,\ldots,\theta_k$ if and only if its associated simplicial complex $\G=\G_1\oplus\G_2\oplus\ldots\oplus\G_k$ is the disjoint union of $k$ connected simplicial complicies and the vertex of maximal degree within each complex is a root-label.
\end{proposition}

\begin{proof} 
Let $\T$ be a saturated tree. If no edge labels are
identified, then writing \cref{eq:rootexpansion} as
$c_{\T} = \sum_{i=1}^k \theta_i c_i$
we find that no two $c_{i}$ and $c_{j}$, $i\not=j$,
have any indeterminates in common, $i, j=1,\ldots,k$. Thus, we can split the set of
atomic monomials $\boldtheta^{\alpha_i}$, $i=1,\ldots,n$, into $k$
disjoint sets, each given by the monomial terms in one
$\theta_ic_i$. This gives us the disjoint union
of $\G=\G_1\oplus\G_2\oplus\ldots\oplus\G_k$. By the linear expansion
of the interpolation polynomial,  the vertex $\theta_i$ is connected to every
other monomial in $\G_i$. 
It is thus of highest degree in the sense that it has the highest number of emanating edges. 
For if in $\G_i$ there was a second vertex
$\theta_j$, $i\not=j$, of equally high degree then both $\theta_i$
and $\theta_j$ would divide every monomial in that subset. But by definition a sequence of single
edges, here labeled $\theta_i$ and $\theta_j$, is not possible.

Conversely, assume we have a set of monomials belonging to an event tree. 
Then the associated simplicial complex is the disjoint union of simplicial complicies $\G=\G_1\oplus\G_2\oplus\ldots\oplus\G_k$ where each $\G_i$ has a vertex  $\theta_i$ of highest degree, $i=1,\ldots,k$. Thus, we can write the corresponding interpolating polynomial in the form \cref{eq:rootexpansion}. Because no $\G_i$ is connected to any $\G_j$ for $i\not=j$, the terms belonging to one sub-simplicial complex have no indeterminates in common with those belonging to the other. Thus the subtrees rooted after the root do not have any labels in common. Therefore the original tree is saturated.
\end{proof}

The proposition enables us to use this simplicial complex representation of an interpolating polynomial to quickly decide whether or not the corresponding labeled event tree is saturated. Thus, by \cref{rem:saturated}, we will know whether or not we need to check for different nested representations of its interpolating polynomial, or whether or not any representation that is discovered is unique . If a tree is saturated, we can then resize it to a simpler graphical representation as in \cref{example:minrepre}.
\medskip

\begin{figure}[tb]
\centering
\includegraphics[scale=1]{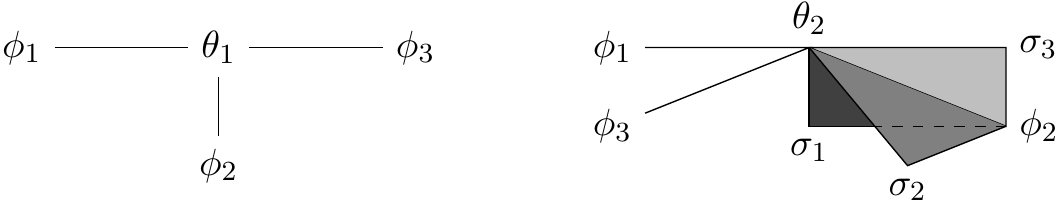}
\caption{The simplicial complex for the staged tree with nested representation ${\cT=\theta_1(\phi_1+\phi_2+\phi_3)+ \theta_2(\phi_1+\phi_2(\sigma_1{+}\sigma_2{+}\sigma_3)+ \phi_3)}$ given in~\cref{eq:nestingbsp1} 
is the direct sum of the two simplicial complexes above. 
The three triangles in the right hand complex with vertices $\theta_2\phi_2 \sigma_i$, $i=1,2,3$, correspond to the root-to-leaf paths of length three. 
}\label{fig:simplicialcomplex} 
\end{figure}

The other natural representation of these monomials is via an incidence matrix. 
Let $\mathcal T=\pt$ be a labeled event tree with monomials  $\theta_{1}^{\alpha_{1,j}}\theta_2^{\alpha_{2,j}}\cdots\theta_d^{\alpha_{d,j}} =\boldtheta^{\alpha_j}$, for $\alpha_j\in\mathbb{Z}_{\geq 0}^d$ and $j=1,\ldots,n$. 
The interpolating polynomial of $\mathcal T$ can be visualized by 
a $d\times n$ matrix $A_{\T}=(a_{ij})_{ij} $ with integer non-negative entries such that
\begin{equation}
a_{ij}=\begin{cases} m \quad&\text{if } \theta_i^m \text{ divides } \boldtheta^{\alpha_j} \text{ and } m\in\mathbb{N} \text{ is maximal}\\
0 &\text{otherwise.}
\end{cases}
\end{equation}
If the atomic monomials in $\mathcal T$ are square-free then $A_{\T}$ is a matrix with entries $0$ or $1$. 
The matrix $A_{\T}$ codes a number of properties of the atomic monomials of $\T$.
In particular, every column encodes those indeterminates which divide the associated monomial, so column sums are the degree of the monomial indexing the column. Every row sum codes the number of monomials which are divided by a certain indeterminate. In order for a set of monomials to be associated to a tree, we need that
\begin{equation}
\summ{i=1,\ldots,d}a_{il}~<~\summ{j=1,\ldots,n}a_{kj}
\end{equation}
for all pairs of $k,l$.  This follows from \cref{prop:characterisation}.2. Submatrices of $A_{\T}$ can easily be associated to subtrees of $\T$. For instance for a subtree $\T_v\subseteq\T$ rooted after an edge $(\cdot,v)$ labeled $\theta_i$, we cancel all rows $a_{i\cdot}$ and all columns $a_{\cdot j}$ from the matrix which include an entry $a_{ij}=0$. The remaining matrix $A_{\T,i}=A_{\T_v}$ is  then the incidence matrix of~$\T_v$.
 
For example, the incidence matrix $A_{\T}$ for the interpolating polynomial $\cT$ in \cref{eq:exatomic} 
of the trees in \cref{fig:nesting} is 
\begin{equation}
\bordermatrix{
& \theta_1 \phi_1 & \theta_1 \phi_2 & \theta_1 \phi_3 &
\theta_2 \phi_1 & \theta_2 \phi_3 &  \theta_2 \phi_2 \sigma_1 &  \theta_2 \phi_2 \sigma_2 &  \theta_2 \phi_2 \sigma_3 \cr
\theta_1& 1 & 1 & 1 & 0 & 0 & 0 & 0 & 0  \cr
\theta_2 & 0 & 0 & 0 & 1 & 1 & 1 & 1 & 1  \cr
\phi_1  & 1 & 0 & 0 & 1 & 0 & 0 & 0 & 0 \cr
\phi_2 & 0 & 1 & 0 & 0 & 0 & 1 & 1 & 1 \cr
\phi_3 & 0 & 0 & 1 & 0 & 1 & 0 & 0 & 0 \cr
\sigma_1 & 0 & 0 & 0 & 0 & 0 & 1 & 0 & 0  \cr
\sigma_2 & 0 & 0 & 0 & 0 & 0 & 0 & 1 & 0 \cr
\sigma_3 & 0 & 0 & 0 & 0 & 0 & 0 & 0 & 1 }
\end{equation}
The sum of the first two rows in this matrix is a vector with all entries equal to one and the labels indexing these first two rows are root-floret labels.  
This is not by chance. In fact, the full tree can be retrieved by splitting the set of columns into those which have one in the first row or in the second row and proceeding recursively. 
 This procedure can be turned into a matrix version of the {\tt StagedTree} algorithm.
 
This matrix representation enables us to link model representations given by labeled or staged trees to log-linear models and well-known results in algebraic stiatistics~\citep{geigermeeksturm:algebra}.

\FloatBarrier
\section{An application}\label{sect:CHDS}

In this section we will apply the algorithm presented in \cref{sect:polytograph} to determine the full polynomial equivalence class of a staged tree representing the best fitting model inferred from a real-world dataset.
The work of~\cite{chds} provides an early analysis of what we will refer to as \enquote{the Christchurch dataset}. These data have been collected on a cohort of nearly one thousand children over the course of thirty years and include measurements of a number of possibly relevant factors to determine the likelihood of child illness. These measurements can be grouped into the very broad categories of socio-economic background and number of life events -- like divorce of its parents or death in the family -- of a child, with respective states \enquote{high}, \enquote{average} and \enquote{low}. The state of health of a child is then assessed as hospital admission \enquote{yes} or \enquote{no}~\cite{lorna:BNceg}.

An MAP algorithm running on the Christchurch
dataset determined the highest scoring staged tree
representation among those which had all vertices that are in the same stage also at the same depth~\cite{cowelljim:causal}. Later,~\cite{meins} found a statistically equivalent but graphically simpler
representation with no saturated subtrees. This staged tree $\pt$ is shown in \cref{fig:chds_resized}.
Here, socio-economic background of a child has been modified to a measure of the access to credit which can be high ($++$), moderately high ($+-$ or $-+$) or low ($--$). The colouring of the staged tree then indicates a number of interesting conditional independence statements. For instance, the red stages on the first level of the tree state that the likelihood of hospital admission was inferred to be the same for all children from a family with high or moderately high access to credit. The blue stages on the subsequent level add that the number of life events of a child is independent of it being admitted to hospital given that its family's access to credit was high, but different given that its access to credit was low. From the green stages we can see that for children with moderate access to credit the likelihood of a certain quantity of life events is not independent of admission to hospital.

The order of events depicted by the staged tree in \cref{fig:chds_resized} suggests that the number of life events of a child might be a putative cause of its admission to hospital. The analysis of~\cite{cowelljim:causal,meins} then showed that in fact when keeping the original problem variables intact across the class of staged trees which are statistically equivalent to $\pt$, this order is preserved. This interpretation of the tree's directionality thus seems to be supported by the Christchurch data.
\medskip

\begin{figure}
\centering
\subfloat[The original staged tree $\pt$ from Fig.~4(b) in~\cite{meins}. See \cref{eq:equiv0}.]{
\includegraphics[scale=1]{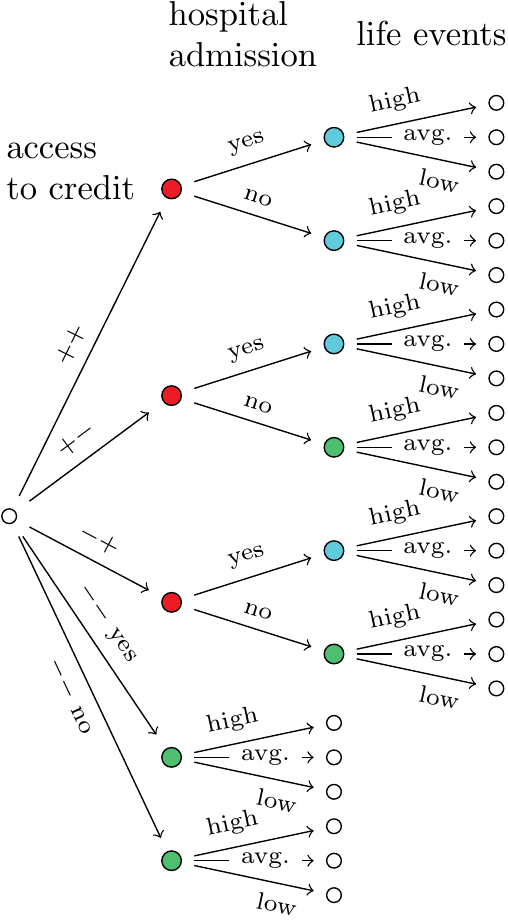}
\label{fig:chds_resized}}
\qquad
\subfloat[A staged tree $\pt_1$ with nested {representation}~\cref{eq:equiv1}.]{
\includegraphics[scale=1]{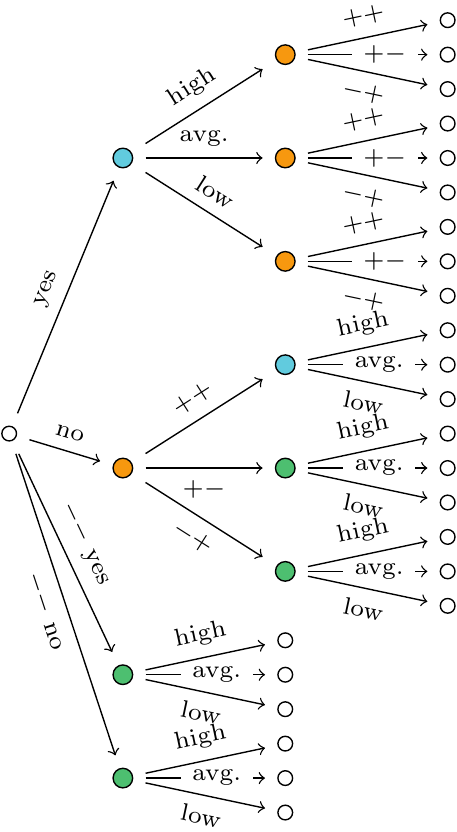}
\label{fig:chds_equiv1}}

\subfloat[A staged tree $\pt_2$ with nested {representation}~\cref{eq:equiv2}.]{
\includegraphics[scale=1]{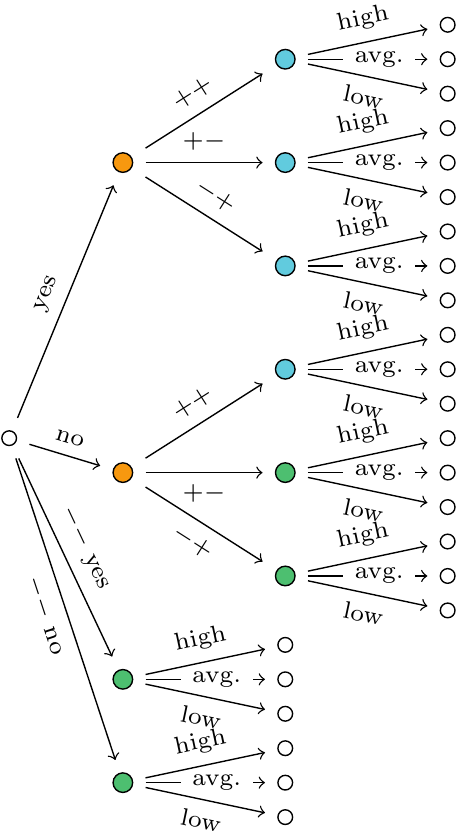}
\label{fig:chds_equiv2}}
\qquad
\subfloat[A staged tree $\pt_3$ with nested {representation}~\cref{eq:equiv3}.]{
\includegraphics[scale=1]{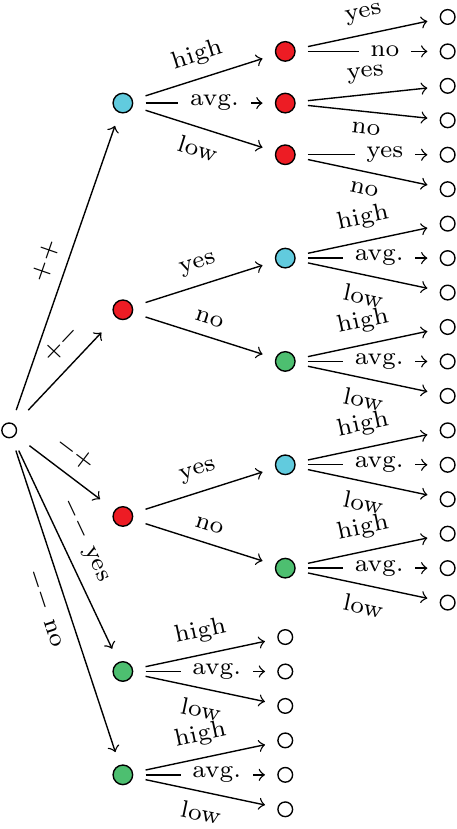}
\label{fig:chds_equiv3}}
\caption{All four elements of the polynomial equivalence class
of $c_\T$ in \cref{eq:chds}.}
\label{fig:chds}
\end{figure}

\begin{table}[t]
\centering
\begin{tabular}{l l l}
stage colour & label & interpretation\\
\hline
 & $(a_1,a_2,a_3,a_4,a_5)$ & access to credit: $++,\ldots,--$\\
\cellcolor{Red} & $(h_1,h_2)$ & hospital admission: yes or no\\
\cellcolor{SkyBlue} & $(l_1,l_2,l_3)$ & number of life events: high, average or low\\
\cellcolor{Green!80} & $(l_3,l_4,l_5)$ & number of life events: high, average or low\\
\cellcolor{YellowOrange} & $(a_1,a_2,a_3)$ & access to credit: $++,+-,-+$
\end{tabular}
\caption{The labels of the staged trees in \cref{fig:chds}, used in the interpolating polynomial~\cref{eq:chds}.}\label{tbl:chds}
\end{table}

We will now use the algorithm \texttt{StagedTrees} in \cref{sect:algorithm} to automatically determine the polynomial equivalence class of $\T=\pt$. To this end we first specify the interpolating polynomial for the tree in \cref{fig:chds_resized}, using labels as specified in \cref{tbl:chds}:
\begin{equation}\label{eq:chds}
\begin{split}
\cT(\boldsymbol{a},\boldsymbol{h},\boldsymbol{l})
~=~&~a_1h_1l_1+a_1h_1l_2+a_1h_1l_3+a_1h_2l_1+a_1h_2l_2+a_1h_2l_3\\
&+a_2h_1l_1+a_2h_1l_2+a_2h_1l_3+a_2h_2l_4+a_2h_2l_5+a_2h_2l_6\\
&+a_3h_1l_1+a_3h_1l_2+a_3h_1l_3+a_3h_2l_4+a_3h_2l_5+a_3h_2l_6\\
&+a_4l_4+a_4l_5+a_4l_6+a_5l_4+a_5l_5+a_5l_6
\end{split}
\end{equation}
where $\boldsymbol{a}=(a_1,a_2,a_3,a_4,a_5)$, $\boldsymbol{h}=(h_1,h_2)$ and $\boldsymbol{l}=(l_1,l_2,l_3,l_4,l_5,l_6)$ are the respective (conditional) probabilities of different degress of \textbf{a}ccess to credit, \textbf{h}ospital admission and numbers of \textbf{l}ife events, read from left to right and from top to bottom along the root-to-leaf paths of $\T$.

Running \texttt{StagedTrees},
we find precisely four different nested {representation}s of $\cT$. These are:
\begingroup
\allowdisplaybreaks
\begin{align}
r_0(\cT)~=~&~
a_1(h_1(l_1+l_2+l_3)+h_2(l_1+l_2+l_3))\label{eq:equiv0}\\
&+a_2(h_1(l_1+l_2+l_3)+h_2(l_4+l_5+l_6))\notag\\
&+a_3(h_1(l_1+l_2+l_3)+h_2(l_4+l_5+l_6))\notag\\
&+a_4(l_4+l_5+l_6)+a_5(l_4+l_5+l_6)\notag\\
r_1(\cT)~=~&~
h_1(l_1(a_1+a_2+a_3)+l_2(a_1+a_2+a_3)+l_3(a_1+a_2+a_3)) \label{eq:equiv1}\\
&+h_2(a_1(l_1+l_2+l_3)+a_2(l_3+l_4+l_5)+a_3(l_3+l_4+l_5))\notag\\
&+a_4(l_4+l_5+l_6)+a_5(l_4+l_5+l_6)\notag\\
r_2(\cT)~=~&~
h_1(a_1(l_1+l_2+l_3)+a_2(l_1+l_2+l_3)+a_3(l_1+l_2+l_3)) \label{eq:equiv2}\\
&+h_2(a_1(l_1+l_2+l_3)+a_2(l_3+l_4+l_5)+a_3(l_3+l_4+l_5))\notag\\
&+a_4(l_4+l_5+l_6)+a_5(l_4+l_5+l_6)\notag\\
r_3(\cT)~=~&~
a_1(l_1(h_1+h_2)+l_2(h_1+h_2)+l_3(h_1+h_2)) \label{eq:equiv3}\\
&+a_2(h_1(l_1+l_2+l_3)+h_2(l_4+l_5+l_6))\notag\\
&+a_3(h_1(l_1+l_2+l_3)+h_2(l_4+l_5+l_6))\notag\\
&+a_4(l_4+l_5+l_6)+a_5(l_4+l_5+l_6)\notag
\end{align}%
\endgroup
where for now $r_i$ denotes one fixed order of summation in a nested representation, $i=0,1,2,3$.

By \cref{prop:nesting2}, $r_0(\cT)$ is the nested factorisation of $\pt$. In \cref{fig:chds_equiv1} we  have drawn the staged tree $\pt_1$ corresponding to the representation $r_1(\cT)$, in \cref{fig:chds_equiv2} the staged tree $\pt_2$ corresponding to $r_2(\cT)$ and in \cref{fig:chds_equiv3} the staged tree $\pt_3$ corresponding to $r_3(\cT)$. These staged trees are the only labeled event trees with the above interpolating polynomial on which \mbox{sum-to-$1$} conditions imposed on florets induce a probability distribution over the depicted atoms. So in \cref{fig:chds} we see all four elements of the polynomial equivalence class of $\pt$. By \cref{def:equivalence}, these staged trees all represent the same underlying model. So we can now analyse the orders in which the same events are depicted across different graphs.

Because in \cref{fig:chds_resized} and~\ref{fig:chds_equiv2} all vertices in the same stage are also at the same distance from the leaves, we can in this case assign an interpretation to each such level of the tree. So in \cref{fig:chds_resized} the first level of $\pt$ depicts all states of the random variables access to credit, the second level depicts all states of the random variable hospital admission and the third and last level depicts all states of the random variable life events. Now this interpretation has been reversed in \cref{fig:chds_equiv2}. In $\pt_2$, the third level still depicts life events but the first two levels have been interchanged. The first level now represents the states of a joint random variable \enquote{hospital admission} and \enquote{hospital admission having low access to credit}. The second level then depicts access to credit with states \enquote{high} and \enquote{moderately high}. So because both $\pt$ and $\pt_2$ represent the same model with $\pt$ showing access to credit before hospital admission and $\pt_2$ reversing that order, we cannot hypothesize a putative causal relationship on these (conditionally independent) variables: see~\cite{ours} for a more thorough presentation of this very subtle point.

It is less straightforward to assign a meaning in terms of problem variables to the staged trees in \cref{fig:chds_equiv1} and~\ref{fig:chds_equiv3}. However, we can still see when comparing $\pt_1$ with $\pt_2$ or $\pt$ with $\pt_3$ that only for children from a family with high access to credit is the order of hospital admission and life events reversible. In all other circumstances the model depicts hospital admission before life events. As in~\cite{cowelljim:causal,meins}, we therefore might want to assign this a putative causal interpretation. 

\section*{Acknowledgments} 
Christiane G\"orgen was supported by the EPSRC grant  EP/L505110/1. Part of this research was supported through the programme \enquote{Oberwolfach Leibniz Fellows} by the Mathematisches Forschungsinstitut Oberwolfach in 2017. During some of this development Jim Q.~Smith was supported by The Alan Turing Institute under the EPSRC grant EP/N510129/1.


\FloatBarrier
\appendix
\section*{Appendix} 

\subsection*{Square-free monomial ideals}
We summarize here the notions from commutative algebra which have been
mentioned in this paper.

Given a non-zero polynomial $f\in \mathbb R[x_1,\dots,x_d]$, with
\textbf{coefficients} in $\mathbb R$ and \textbf{indeterminates} (or
variables) $x_1,\dots,x_d$, $f$ is uniquely written as $f =
\sum_{i=1}^s \beta_i t_i$, with coefficients $b_i\ne0$, and
\textbf{power-products} (or terms, or monomials) $t_i =
x_1^{\alpha_{i,1}}\cdots x_d^{\alpha_{i,d}}$ all distinct, for every
$i = 1,\dots,s$.

The \textbf{support} of a polynomial $f$
is the set of the power-products actually occurring in $f$.
With the notation above, $\supp(f) = \{t_i\mid i = 1,\dots,s\}$.

An \textbf{ideal} generated by a set of polynomials,
say $I = \ideal{f_1, \dots, f_k}$,
is the set of all linear combinations with polynomial coefficients,
i.e. $I = \{ g_1f_1 +\dots + g_kf_k \mid g_i\in \mathbb
R[x_1,\dots,x_d] \text{ for } i=1,\ldots,k \}$.
In particular, if all $f_i$'s are power-products, $I$ is called a 
\textbf{monomial ideal}.
If a power-product has all exponents in  $\{0,1\}$, it is said 
\textbf{square-free}, and an ideal generated by square-free
power-products is called \textbf{square-free monomial ideal}.

Given a monomial ideal $I$, a \textbf{minimal prime}  of $I$ is
an ideal ${\cal P}$ generated by a subset of the indeterminates
$\{x_1,\dots,x_d\}$ such that $I$ is contained in ${\cal P}$, but is not
contained in any ideal generated by a subset of the generators of
${\cal P}$ (used in \cref{thm:minimal}).

An ideal is \textbf{primary} if $f g \in I$ implies either $f \in I$ or
 some power $g^m\in I$ (for some integer $m > 0$).
All ideals in $\mathbb R[x_1,\dots,x_d]$ admit a \textbf{primary
  decomposition}, i.e. may be written as an intersection of primary
ideals.  In the particular case of interest in this paper, a
square-free monomial ideal has primary decomposition $I={\cal
  P}_1\cap\dots{\cal P}_\ell$, where the primary ideals ${\cal P}_i$
are indeed the minimal primes of $I$.  In general, the \textbf{prime
  decomposition} of an ideal is given by the minimal primes of the
ideal (used in \cref{ex:non-square-free}), and is the primary
decomposition of the radical of the ideal.

In general, computing the primary decomposition of a polynomial ideal
is quite difficult, but for monomial ideals the operations are a lot
easier.  In particular, for square-free monomial ideals there is a
very simple and efficient algorithm called \textbf{Alexander Dual}.

\bibliographystyle{plain}
\small{
\setlength{\bibsep}{2pt plus 0.3ex}
\bibliography{christianesbib}}

\paragraph{Authors' addresses}~\\
Christiane G\"orgen, \texttt{goergen@mis.mpg.de}, Max-Planck-Institute for Mathematics in the Sciences, Leipzig, Germany.\\
Anna Bigatti,  \texttt{bigatti@dima.unige.it}, Dipartimento di Matematica, Universit\`a degli Studi di Genova, 16146 Genova, Italy.\\
Eva Riccomagno, \texttt{riccomagno@dima.unige.it}, Institute of Intelligent Systems for Automation, National Research Council, Italy; and Dipartimento di Matematica, Universit\`a degli Studi di Genova, 16146 Genova, Italy.\\
Jim Q.~Smith, \texttt{j.q.smith@warwick.ac.uk}, Department of Statistics, University of Warwick, Coventry CV5 7AL, U.K.; and The Alan Turing Institute, British Library, 96 Euston Road, NW1 2DB London, U.K..
\end{document}